\documentclass[11pt]{article}
\usepackage{epigamath}

\usepackage[notext]{kpfonts}
\usepackage{baskervald}

\setpapertype{A4}

\usepackage[english]{babel}

\usepackage[dvips]{graphicx}     

\removelength{1cm}

\title{Rigidity properties of holomorphic Legendrian singularities}
\titlemark{Holomorphic Legendrian singularities}
\author{Jun-Muk Hwang}
\authoraddresses{
\authordata{Jun-Muk Hwang}{\firstname{Jun-Muk} \lastname{Hwang}\\
\institution{Korea Institute for Advanced Study, Hoegiro 87, Seoul 02455, Korea}\\
\email{jmhwang@kias.re.kr}}
}
\authormark{J.-M. Hwang}
\date{\vspace{-5ex}} 
\journal{\'Epijournal de G\'eom\'etrie Alg\'ebrique} 
\acceptation{Received by the Editors on May 10, 2018, and in final form on August 13, 2019. \\ Accepted on October 25, 2019.}

\newcommand{\articlereference}{Volume 3 (2019), Article Nr. 18}

\acknowledgement{The author is supported
by National Researcher Program 2010-0020413 of NRF}

\usepackage[all]{xy}

\allowdisplaybreaks

\newdimen\origiwspc
\origiwspc=\fontdimen2\font

\newcommand{\extendspace}[1]{\fontdimen2\font=1.2ex #1 \fontdimen2\font=\origiwspc}

\usepackage{float}

\numberwithin{equation}{section}

\makeatletter

\renewcommand{\p@equation}{\arabic{section}.\arabic{equation}\expandafter\@gobble}
\makeatother

\newtheorem{thm}{Theorem}[section]
\newtheorem{theorem}[thm]{Theorem}

\newtheorem{lemma}[thm]{Lemma}

\newtheorem{proposition}[thm]{Proposition}

\newtheorem{notation}[thm]{Notation}

\newtheorem{definition}[thm]{Definition}

\newtheorem{example}[thm]{Example}

\newcommand{\sD}{{\mathcal D}}

\newcommand{\sI}{{\mathcal I}}

\newcommand{\sL}{{\mathcal L}}

\newcommand{\sO}{{\mathcal O}}

\newcommand{\sS}{{\mathcal S}}

\newcommand{\sY}{{\mathcal Y}}

\newcommand{\C}{{\mathbb C}}

\newcommand{\BP}{{\mathbb P}}

\newcommand{\fsl}{{\mathfrak s}{\mathfrak l}}

\newcommand{\fsp}{{\mathfrak s}{\mathfrak p}}

\newcommand{\cont}{{\mathfrak c}{\mathfrak o}{\mathfrak n}{\mathfrak t}}

\def\Sym{\mathop{\rm Sym}\nolimits}

\def\Hom{\mathop{\rm Hom}\nolimits}

\begin{document}


\maketitle



\begin{prelims}


\def\abstractname{Abstract}
\abstract{We study the singularities of Legendrian subvarieties of contact manifolds in the complex-analytic category and prove two rigidity results. The first one is that Legendrian singularities with reduced tangent cones are contactomorphically biholomorphic to their tangent cones.  This result is partly motivated by a problem on Fano contact manifolds. The second result is the deformation-rigidity of normal Legendrian singularities, meaning that any holomorphic family  of normal Legendrian singularities is trivial, up to contactomorphic biholomorphisms of germs. Both results are proved by exploiting the relation between infinitesimal contactomorphisms and holomorphic sections of the natural line bundle on the contact manifold.}

\keywords{Legendrian singularities, Lagrangian cone, tangent cone}

\MSCclass{58K40, 58K60, 53D10, 14B07}

\vspace{0.15cm}

\languagesection{Fran\c{c}ais}{%

\vspace{-0.05cm}
{\bf Titre. Propri\'et\'es de rigidit\'e des singularit\'es legendriennes holomorphes} \commentskip {\bf R\'esum\'e.} Nous \'etudions les singularit\'es des sous-vari\'et\'es legendriennes des vari\'et\'es de contact dans la cat\'egorie analytique complexe et montrons deux r\'esultats de rigidit\'e. Le premier affirme que les singularit\'es legendriennes ayant un c\^one tangent r\'eduit sont contactomorphiquement biholomorphes \`a ce dernier. Ce r\'esultat est en partie motiv\'e par un probl\`eme concernant les vari\'et\'es de contact qui sont \'egalement de Fano. Le second r\'esultat consiste en la rigidit\'e par d\'eformation des singularit\'es legendriennes normales, ceci signifiant que toute famille holomorphe de singularit\'es legendriennes normales est triviale, \`a un germe de biholomorphisme contactomorphe pr\`es. Ces deux r\'esultats sont d\'emontr\'es en exploitant la relation entre les contactomorphismes infinit\'esimaux et les sections holomorphes d'un fibr\'e en droites naturel sur la vari\'et\'e de contact.}

\end{prelims}


\newpage

\setcounter{tocdepth}{1} \tableofcontents

\section{Introduction} 

Singularities of Legendrian varieties in contact manifolds have been studied  in singularity theory or symplectic/contact geometry, often in differentiable or real analytic categories. In this article, we study them in the complex-analytic category. Methods of algebraic geometry can be applied more efficiently to holomorphic Legendrian singularities. Using this approach, we present two rigidity results on Legendrian singularities.  Let us start with precise definitions of the terms we  use.

\begin{definition}\label{d.contact} {\rm
Let $M$ be a complex manifold of  dimension $2m+1$ for a positive integer $m$. A subbundle $D \subset T_M$ of rank $2m$   in the holomorphic tangent bundle $T_M$ is called a {\em contact structure } on $M$ if the Frobenius bracket homomorphism
$\wedge^2 D \to T_M/D$ is nondegenerate at every point of $M$. A complex manifold $M$ equipped with a contact structure $D$ is called a {\em contact manifold}.
A biholomorphic map $\varphi: M_1 \to M_2$ between two contact manifolds $(M_1, D_1)$ and $(M_2, D_2)$ is {\em contactomorphic} (equivalently, a {\em contactomorphism})
if ${\rm d} \varphi (D_1) = D_2$.
}
\end{definition}

\begin{definition}\label{d.Legendre} {\rm
Let $(M,D)$ be a contact manifold of dimension $2m+1$.  An analytic subvariety $Z \subset M$ is {\em Legendrian} if $\dim Z = m$ and $T_{Z,x} \subset D_x$ for each nonsingular point $x$ of $Z$. This implies that $T_{Z,x} \subset D_x$ is isotropic with respect to the bracket $\wedge^2 D_x \to T_{M,x}/D_x$. A nonsingular Legendrian subvariety is called a {\em Legendrian submanifold}. The germ of a point on a Legendrian subvariety $x \in Z$ in a contact manifold $(M,D)$ is called a {\em Legendrian singularity}. Two Legendrian singularities $x_1 \in Z_1 \subset (M_1, D_1)$
and $x_2 \in Z_2 \subset (M_2, D_2)$, where $Z_i$ is a Legendrian subvariety in a contact manifold $(M_i, D_i),\, i=1,2,$ are {\em contactomorphic}, if there exist  open neighborhoods
$U_i \subset M_i$ of $x_i,\, i=1,2$, and a contactomorphism $\varphi: U_1 \to U_2$ such that
$$\varphi(x_1)= x_2 \ \mbox{ and }  \varphi(Z_1 \cap U_1) = Z_2 \cap U_2.$$ }\end{definition}

Are there many interesting examples of Legendrian singularities? The following construction provides lots of them.

\begin{example}{\rm
For a complex manifold $X$, the projectivized cotangent bundle $M= \BP T^*_{X}$ has a natural contact structure (see e.g. Example 1.2 B of \cite{AG} or Example 2.2 of \cite{LB2}). For any complex analytic subvariety $Y \subset X$, its conormal variety
$Z_Y \subset M$, the closure of the projectivized conormal bundle of
the smooth locus of $Y$, is a Legendrian subvariety. The conormal variety $Z_Y$ is usually (but not always) singular when $Y$ is singular. When $Y \subset X$ is a hypersurface, the conormal variety $Z_Y$ is the  Nash blowup of $Y$. } \end{example}

In Section \ref{s.cone}, we give another class of examples of Legendrian singularities,
those arising from Lagrangian cones.

 Our first rigidity result is in terms of tangent cones.
Recall (see Chapter 3, Section 3 of \cite{Mu}) that for an analytic subvariety $Z$  of a complex manifold $M$ and a point $x \in Z$,  if $I \subset  \sO_{M,x}$ is the ideal of the germ of $Z$ at $x$, then the tangent cone $TC_{Z,x}$ is the  subscheme of the Zariski tangent space $T_{Z,x}$ defined by the ideal generated by lowest-order terms of the Taylor expansions of  elements of $I$ at $x$. Roughly speaking, the tangent cone of a singular variety is the lowest order approximation of the singularity. It seldom determines the singularity. Remarkably, a Legendrian singularity is determined by the tangent cone, if the tangent cone is reduced.

\begin{theorem}\label{t.cone}
 A Legendrian singularity  is contactomorphic to the germ at the origin of its tangent cone  if the tangent cone is reduced. More precisely, a Legendrian singularity is a Lagrangian cone singularity (in the sense of Definition \ref{d.LagCone}) if  and only if its
tangent cone is reduced. \end{theorem}

This is proved in Section \ref{s.reduced}. Of course, the reducedness of the tangent cone is a strong requirement. There are many examples of Legendrian singularities with non-reduced tangent cones: for instance, cuspidal Legendrian curves  discussed in Section 4 of \cite{Zh}. 
One motivation for Theorem \ref{t.cone} comes from the study of Fano contact manifolds.
In his investigation \cite{Ke} of Fano contact manifolds, Kebekus studied a certain Legendrian singularity $x \in Z$ ($ Z= {\rm locus}(H_x)$ in the notation of \cite{Ke}). He showed that the projectivized tangent cone $\BP TC_{Z,x}$ is nonsingular and asserted that the singularity $z \in Z$  is biholomorphic to the germ of a Lagrangian cone at $0$. We believe that the latter assertion, if it is true, would have  significant consequences in the study of Fano contact manifolds. But its proof given  in Section 6.1 of \cite{Ke} had a gap.  Theorem \ref{t.cone} has grown out of our attempt to remedy this gap. But it is not strong enough to fix it, as the smoothness of $\BP TC_{Z,x}$ does not imply that $TC_{Z,x}$ is reduced. A  technical difficulty here arises from the fact that $Z$ is (a priori) not normal.

In fact,  Legendrian singularities are  usually not normal and their normalizations cannot be realized as Legendrian singularities.  This can be seen from the following result on the deformation-rigidity of normal Legendrian singularities, which is proved in Section \ref{s.normal}.

\begin{theorem}\label{t.normal}
Let $\Delta$ be a neighborhood of the origin $0$ in $\C$.
Let $(M,D)$ be a contact manifold and consider $\{Z_t \subset M, t \in \Delta\}$ a holomorphic family of Legendrian subvarieties parametrized by $\Delta$. Assume that $Z_t$ is normal for
every $t \in \Delta$.
Then for any $x_0 \in Z_0$, there exist a neighborhood $0 \in \Delta' \subset \Delta$ and a holomorphic arc  $$\{ x_t \in Z_t, t \in \Delta'\}$$ such that the Legendrian singularities $x_0 \in Z_0 \subset M$ and $x_t \in Z_t \subset M$ are contactomorphic for each $t \in \Delta'$.
\end{theorem}

Theorem \ref{t.normal} suggests that it might be possible to classify normal Legendrian singularities to some extent. In fact, normal Legendrian singularities are not easy to find.  Some normal Legendrian singularities are described in Example \ref{e.sub}.

Normal Legendrian singularities are interesting from another viewpoint.
The following theorem says that a Legendrian singularity, unless it is nonsingular,  has nonzero torsion differentials.
This is a special case of a stronger result, Theorem 2.5 in \cite{Zh}. For the reader's convenience, we give an elementary proof (different from the one in \cite{Zh}) at the end of Section \ref{s.torsion}.

\begin{theorem}\label{t.torsion}
For a Legendrian singularity $x \in Z \subset (M,D),$ let $\theta$ be a germ of 1-form at $x \in M$ defining $D$. Then $x$ is a nonsingular point of $Z$  if and only if
$\theta|_Z$ is zero in the space of K\"ahler differentials $\Omega_{Z,x}.$
 \end{theorem}

By Theorem \ref{t.torsion}, normal Legendrian singularities provide examples of normal singularities with explicit nonzero torsion differentials.
We mention that some examples of normal singularities with
nonzero torsion differentials
 were given in \cite{GR} by cohomological methods.
 One of their examples, the cone over the twisted cubic curve ($d=3$ in Proposition 4.1 of \cite{GR}), is a Legendrian singularity
 in Example \ref{e.sub}.

\paragraph{\bf Acknowledgment.}
I am grateful to  Manfred Lehn and Duco van Straten for showing me the subtle difference between tangent cones and projectivized tangent cones. I would like to thank Go-o Ishikawa  for discussions on Legendrian singularities.

\section{Torsion differentials of Legendrian singularities}\label{s.torsion}

\begin{notation}\label{n.standard}{\rm
Fix a linear coordinate system $(x_1, \ldots, x_{2m+1})$ on $\C^{2m+1}$.  Set
\begin{eqnarray*}  \theta &:=& \sum_{i=1}^m (x_{m+i} {\rm d} x_i - x_i {\rm d} x_{m+i}) - {\rm d} x_{2m+1}  \end{eqnarray*}
Then $\theta =0$ defines a contact structure on $\C^{2m+1}$, which we call the {\em standard contact structure}. }\end{notation}

 By Darboux theorem  (Chapter 4, Section 1.1 of \cite{AG}), any contact structure is locally equivalent to the standard contact structure. Thus when studying a Legendrian singularity $x \in Z \subset (M, D)$, we may
 assume that $M$ is a neighborhood of $\C^{2m+1}$ and $D$ is the standard contact structure. We remark that in many references (like \cite{AG} or \cite{Kb}) the form $\sum_{i=1}^m x_i {\rm d} x_{m+i} - {\rm d} x_{2m+1}$ is used as the standard form. When algebro-geometric tools are used, however, our choice $\theta$ is more convenient because it is the expression of a contact structure on $\BP^{2m+1}$  in affine coordinates.

The following is a standard result in contact geometry. It is essentially given in p.~79 of \cite{AG} or pp.~30-31 of \cite{Kb}. As our standard form $\theta$ is slightly different from theirs, we recall the proof for readers' convenience.

\begin{theorem}\label{t.Kobayashi}
In Notation \ref{n.standard}, let $U$ be a neighborhood of $0 \in \C^{2m+1}$.  For a holomorphic function $f$ on $U$, let $F \subset U$ be the hypersurface defined by $f=0$ and let $v^f$ be
the holomorphic vector field on $U$ defined by 
\begin{align*}
2 v^f & = \sum_{k=1}^m (\frac{\partial f}{\partial x_{m+k}} - \frac{\partial f}{\partial x_{2m+1}}  x_k ) \frac{\partial }{\partial x_k} + \sum_{k=1}^m (-\frac{\partial f}{\partial x_k} - \frac{\partial f}{\partial x_{2m+1}} x_{m+k})\frac{\partial}{\partial x_{m+k}} \\
& + \left(\sum_{k=1}^m (\frac{\partial f}{\partial x_k} x_k+ \frac{\partial f}{\partial x_{m+k}} x_{m+k} ) - 2f \right) \frac{\partial}{\partial x_{2m+1}}. 
\end{align*}
Then  \begin{itemize} \item[\rm (i)]
$\theta(v^f) = f$; \item[\rm (ii)] $ v^f(f) = - f \frac{\partial f}{\partial x_{2m+1}}$;
  \item[\rm (iii)]
     $v^f$ is zero at a point $y \in F$ if and only if $F$ is singular at $y$ or
     $\theta(T_{F,y}) =0$; \item[\rm (iv)] $v^f$ is tangent to the hypersurface  $F$; \item[\rm (v)] ${\rm Lie}_{v^f} \theta = - \frac{\partial f}{\partial x_{2m+1}} \theta$;  \item[\rm (vi)]
     for any nonsingular point $y \in F$ and any tangent vector $w \in T_{F,y}$ satisfying $\theta(w) =0$, we have  ${\rm d} \theta (v^f(y), w) = 0$
     ; and \item[\rm (vii)]  $v^f$ is tangent to
   any Legendrian subvariety $Z \subset U$ contained in the hypersurface $ F.$ \end{itemize} \end{theorem}

\begin{proof} (i), (ii), (iii) can be checked by straightforward calculation. (iv) is immediate from (ii). (v) can be checked from  Cartan formula,
$${\rm Lie}_{v^f} \theta = {\rm d} (\theta(v^f)) + {\rm d}\theta (v^f, \cdot) = {\rm d} f + {\rm d} \theta(v^f, \cdot),$$ and $ {\rm d} \theta = 2 \sum_{k=1}^m {\rm d} x_{m+k} \wedge {\rm d} x_k.$
(vi) follows from  Cartan formula again:  $${\rm d} \theta(v^f, w) = ({\rm Lie}_{v^f}\theta) (w) - {\rm d} f (w) = - \frac{\partial f}{\partial x_{2m+1}} \theta(w) - {\rm d} f(w). $$
 It remains to prove (vii). Pick a nonsingular point $z \in Z$.   The 2-form ${\rm d} \theta$ induces a nondegenerate 2-form on the vector space $D_z$ by the definition of the contact structure. The tangent space $T_{Z,z}$ is an isotropic subspace of $D_z$ with respect to this 2-form ${\rm d} \theta|_{D_z}$ and it is contained in $T_{F,z}$ from $Z \subset F$.
 By (vi), the vector $v^f(z) \in D_z$ satisfies ${\rm d} \theta (v^f(z), T_{Z,z}) =0$.
 Thus the linear span $\langle v^f(z), T_{Z,z} \rangle$ is a subspace of the $2m$-dimensional vector space $D_z$ and is isotropic with respect to the nondegenerate 2-form ${\rm d} \theta|_{D_z}$. Since $\dim T_{Z,z} = m$, we have $v^f(z) \in T_{Z,z}$.   \end{proof}

To see the geometric meaning of Theorem \ref{t.Kobayashi} (vi), it is convenient to recall the notion of Cauchy characteristic of a distribution.

\begin{definition} {\rm Let $ \sD \subset T_X$ be a vector subbundle of corank $1$ on a complex manifold and denote by
$\sigma: \wedge^2 \sD \to T_X/\sD$ the Frobenius bracket tensor.
For each $x \in X$, the Cauchy characteristic of $\sD$ at $x$ is
$${\rm Ch}(\sD)_x := \{ v \in \sD_x, \sigma_x(v, u) = 0 \mbox{ for all } u \in \sD_x \}.$$
In particular, the subbundle $\sD$ is a contact structure if and only if ${\rm Ch}(\sD)_x =0$ for each $x \in X$. } \end{definition}

The following is standard. (1) is straightforward to check and (2) is a special case of  Theorem 2.2 in Chapter 2 of \cite{BCG}.

\begin{lemma}\label{l.Cauchy}
Let $(M, D)$ be a contact manifold and let $X \subset M$ be a nonsingular hypersurface.
Let $X_o \subset X$ be the open subset defined by
$$X_o := \{ x \in X_o, T_{X, x} \neq D_x\}$$ such that $\sD:= D|_{X_o} \cap T_{X_o}$ is a vector subbundle of corank $1$ in $T_{X_o}$. Then
\begin{itemize}
\item[\rm (1)] $\dim {\rm Ch}(\sD)_x =1$ for each $ x \in X_o$. In particular, the Cauchy characteristic ${\rm Ch}(\sD)$ determines a foliation of rank $1$ on $X_o$.
    \item[\rm (2)] For each $x \in X_o$, choose a neighborhood $O \subset X_o$ of $x$
    equipped with a holomorphic submersion $\psi: O \to B$ whose fibers are leaves of the foliation in (1). Then there exists a contact structure $D'$ on $B$ such that
    $\sD_y = ({\rm d}\psi)_y^{-1}(D'_{\psi(y)})$ for each $y \in O$. \end{itemize}
    \end{lemma}

Theorem \ref{t.Kobayashi} (vi) says that the leaves of $v^f$ are the foliation given
by Lemma \ref{l.Cauchy} applied to the nonsingular locus $X$ of the hypersurface $F$.
This is used to prove the  next proposition, which is a direct translation   of Proposition 1 in \cite{Gi} in  symplectic geometry into the setting of contact geometry.

\begin{proposition}\label{p.product}
Let $0 \in Z \subset (\C^{2m+1}, D=(\theta=0))$ be a Legendrian singularity.
Suppose the Zariski tangent space $T_{Z,0} \subset T_{\C^{2m+1},0}$ does not contain $D_0 \subset T_{\C^{2m+1},0}.$  Then there exist \begin{itemize}
 \item[\rm (1)] a holomorphic function $f$ in a neighborhood $U$ of $0$ in $\C^{2m+1}$ defining  a smooth hypersurface $0 \in F \subset U$ with $Z \subset F$
 and $T_{F,0} \neq D_0$;
 \item[\rm (2)] a contact manifold $(M', D')$ with $ \dim M' = 2m-1$; and
    \item[\rm (3)] a submersion $\psi: F \to M'$ whose fibers are leaves of the vector field $v^f$ in the sense of Theorem \ref{t.Kobayashi},
    \end{itemize}
    such that
    \begin{itemize}
    \item[\rm (a)] $D_x \cap T_{F,x} = ({\rm d}_x \psi)^{-1} (D'_{\psi(x)})$ for any $x \in F$; and
        \item[\rm (b)]
     $\psi(Z)$ is a Legendrian subvariety of $(M', D')$ and $Z = \psi^{-1}(\psi(Z))$. \end{itemize}
    \end{proposition}

    \begin{proof}
    The assumption $D_0 \not\subset T_{Z,0}$ implies  the existence of $f$ and $F$ in (1). Using the vector field $v^f$ of Theorem \ref{t.Kobayashi}, we obtain a submersion $\psi: F \to M'$ whose fibers are leaves of $v^f$.  By  (vi) of Theorem \ref{t.Kobayashi} and Lemma \ref{l.Cauchy},
    there exists a contact structure $D'$ on $M'$ satisfying (a). It is clear from (vii) of Theorem \ref{t.Kobayashi} that $Z \cap U = \psi^{-1}(\psi(Z \cap U))$
    and $\psi(Z \cap U)$ is a Legendrian subvariety of $(M', D')$. \end{proof}

By Proposition \ref{p.product}, the proof of  Theorem \ref{t.torsion} is reduced to the next theorem.

\begin{theorem}\label{t.torsion2}
Let $0\in Z \subset \C^{2m+1}$ be the germ of a Legendrian subvariety with respect to the standard contact structure such that $D_0 \subset T_{Z,0}$.
Then the differential $\theta|_Z \in \Omega_{Z,0}$ is not zero.  \end{theorem}

\begin{proof}
We define a weight function  {\rm wt} on $\sO_{\C^{2m+1}, 0}$ and $\Omega_{\C^{2m+1}, 0}$ in the following way. Set
 $${\rm wt} (x_i) = {\rm wt}( {\rm d} x_i)  =  1 \mbox{ for } 1 \leq i \leq 2m \mbox{ and }
 {\rm wt} (x_{2m+1}) = {\rm wt}({\rm d} x_{2m+1}) = 2.$$ Define the weight ${\rm wt} (f)$ of a function $f \in \sO_{\C^{2m+1},0}$ as the weight of the monomial of lowest weight in the Taylor series of $f$ at $0$ and define the weight of elements of $ \Omega_{\C^{2m+1}, 0}$ such that ${\rm wt} ( f \xi) = {\rm wt}(f) + {\rm wt}(\xi)$ for any $\xi \in \Omega_{\C^{2m+1}, 0}.$
 Then  $${\rm wt}(\xi_1 + \xi_2) = {\rm wt}(\xi_1)\ \mbox{ if } \  \xi_1, \xi_2 \in \Omega_{\C^{2m+1}, 0} \mbox{ and } {\rm wt}(\xi_1) < {\rm wt}(\xi_2)$$ and $$ {\rm wt} ({\rm d} f) = {\rm wt}(f) \mbox{ if } f \in \sO_{\C^{2m+1}, 0} \mbox{ and } f(0) =0.$$

 Let $I$ be the ideal of $Z$ in $\sO_{\C^{2m+1}, 0}$.
The space of K\"ahler differentials of $Z$ at $0$ is given by (see e.g. Definition 1.109 of \cite{GLS})
$$\Omega_{Z,0} = \Omega_{\C^{2m+1},0} / ( \sO_{\C^{2m+1}, 0} {\rm d} I + I \cdot \Omega_{\C^{2m+1}, 0} ).$$
Suppose $\theta|_Z =0$, namely,
$$(\mbox{ the germ at $0$ of } \theta)  \  \in \  \sO_{\C^{2m+1}, 0} {\rm d} I + I \cdot \Omega_{\C^{2m+1}, 0}.$$
The condition $D_0 \subset T_{Z,0}$ implies that all elements of $I$ and ${\rm d} I$ have weight at least 2.
Thus all elements of $I \cdot \Omega_{\C^{2m+1},0}$ have weight at least $3$. Since ${\rm wt} (\theta)= 2$, the lowest order term of $\theta$ must be the lowest order term of some element in $\sO_{\C^{2m+1}, 0} \cdot {\rm d} I.$ To have weight 2,
the lowest order term must be ${\rm d}$-exact. But $\theta$ is homogeneous and not ${\rm d}$-exact. A contradiction. \end{proof}

\section{Lagrangian cones as Legendrian varieties}\label{s.cone}

We use the following terms regarding cones.

\begin{definition}\label{d.cone} {\rm
Let $V$ be a complex vector space, which we  regard as an affine space, and let $\Sym^{\bullet}V^*$ be the ring of polynomial functions on $V$.
An {\em affine cone} in $V$ is a subscheme $\sY$ of the affine space $V$ defined by  a homogeneous ideal $I \subset \Sym^{\bullet} V^*$. The corresponding projective subscheme  $\BP \sY \subset \BP V$ will be called the {\em projectivization} of $\sY$. If the subscheme $\sY$ is reduced, i.e., the ideal $I$ is radical  (  $I = \sqrt{I}$ ), we will call it a {\em reduced affine cone}. If $\sY$ is reduced, then so is its projectivization $\BP \sY \subset \BP V$.
But the converse is not always true.} \end{definition}

\begin{definition}\label{d.Lagrange} {\rm
Let $(V, \omega)$ be a symplectic vector space, i.e., a  vector space $V$ equipped with a nondegenerate anti-symmetric 2-form $\omega \in \wedge^2 V^*$. Let $2m$ be the dimension of $V$. An $m$-dimensional reduced affine cone $0\in \sY \subset V$ is called a {\em Lagrangian cone}
 if the restriction of $\omega$ to the nonsingular locus of $\sY$ is zero.} \end{definition}

\begin{lemma}\label{l.Legendre}
In Notation \ref{n.standard}, let $\C^{2m}$ be the hyperplane defined by
$(x_{2m+1} =0) $ and equipped with the symplectic form $$  {\rm d} \theta |_{\C^{2m}} = 2 \sum_{i=1}^m
{\rm d} x_{m+i} \wedge {\rm d} x_i.$$
Let $Z \subset \C^{2m}$ be an $m$-dimensional subvariety.
 When regarded as a subvariety of $\C^{2m+1}$ equipped with the standard contact structure, the variety $Z$  is a Legendrian subvariety if and only if $Z$ is a Lagrangian cone in $(\C^{2m}, {\rm d} \theta |_{\C^{2m}})$ in the sense of Definition \ref{d.Lagrange}. \end{lemma}

\begin{proof}
 We use the radial vector field on $\C^{2m}$ $$\vec{R}:= \sum_{i=1}^m( x_i \frac{\partial}{\partial x_i} + x_{m+i} \frac{\partial}{\partial x_{m+i}} ).$$
In terms of  Theorem \ref{t.Kobayashi}, the radial vector field $\vec{R}$ is the restriction of
$- 2 v^{x_{2m+1}}$ to $\C^{2m}$.

 Assume that $Z \subset \C^{2m}$ is a Lagrangian cone. As it is an affine cone, the radial vector field $\vec{R}$ is tangent to the smooth locus of $Z$.
  It is straightforward to check that the contraction $\vec{R} \lrcorner {\rm d} \theta |_{\C^{2m}} $ is a constant multiple of the 1-form $\theta|_{\C^{2m}}$. It follows that the restriction of $\theta$ to the smooth locus of $Z$ is zero. Thus $Z$ is Legendrian in $\C^{2m+1}$ with respect to $\theta$.

 Conversely, if a subvariety $Z \subset \C^{2m}$ is a Legendrian subvariety of $\C^{2m+1}$ with respect to $\theta$,  then the radial vector field $\vec{R} =- 2 v^{x_{2m+1}}|_{\C^{2m}}$ is tangent to $Z$ by Theorem
 \ref{t.Kobayashi} (vii). Thus $Z$ is a reduced affine cone. Since ${\rm d} \theta$ vanishes on the smooth locus of $Z$,
 it is a Lagrangian cone with respect to ${\rm d} \theta |_{\C^{2m}}$.  \end{proof}

\begin{definition}\label{d.LagCone} {\rm
We say that a Legendrian singularity $x\in Z \subset (M,D)$ is
a {\em Lagrangian cone singularity}, if it is contactomorphic to the germ at $0$ of a Lagrangian cone in $(\C^{2m}, {\rm d} \theta|_{\C^{2m}})$ regarded as a Legendrian subvariety of $(\C^{2m+1}, \theta)$ as in Lemma \ref{l.Legendre}.}
\end{definition}

We skip the proof of the following elementary lemma.

\begin{lemma}\label{l.conesing}
In the setting of Proposition \ref{p.product}, the Legendrian singularity $0 \in Z$ in $(\C^{2m+1}, D)$ is a Lagrangian cone singularity if and only if  the Legendrian singularity $\psi(0) \in \psi(Z)$ in $(M', D')$ is a Lagrangian cone singularity.
\end{lemma}

 There is another way that Lagrangian cones in $(V, \omega)$ give rise to Legendrian subvarieties of a contact manifold. The symplectic form $\omega$ provides the projective space $\BP V$  with the following contact structure (Chapter 4, Section 1.2, Example A in \cite{AG},  Example 2.1 in \cite{LB2}, Section E.1 in \cite{Bu2}).

 \begin{definition}\label{d.PV} {\rm
 For a symplectic vector space $(V,\omega)$, for a point $[v] \in \BP V$ corresponding to
 $ v \in V\setminus 0$, define $D^{\omega}_{[v]} \subset T_{\BP V, [v]}$ by
 $$D^{\omega}_{[v]} := \{ h \in \Hom(\C v, V/(\C v)) = T_{\BP V, [v]}, \omega(v, h(v)) =0\}.$$
 Then the subbundle $D^{\omega} \subset T_{\BP V}$ is a contact structure on $\BP V$.} \end{definition}

The following is well-known (see e.g. Proposition E.2 in \cite{Bu2}).

\begin{proposition}\label{p.projLegendre}
In Definition \ref{d.PV}, for a reduced affine cone $\sY \subset V,$ its projectivization $\BP \sY \subset \BP V$ is a Legendrian subvariety with respect to $D^{\omega}$ if and only if $\sY$ is  Lagrangian with respect to $\omega$.
 \end{proposition}

Legendrian subvarieties of $(\BP V, D^{\omega})$ are studied in \cite{Bu}, \cite{Bu2} and \cite{LM}.

\begin{example}\label{e.sub} {\rm Subadjoint varieties (see Theorem 11 in \cite{LM}) are  Legendrian subvarieties of $({\BP V}, D^{\omega})$ that are homogeneous under the action of the symplectic automorphisms of $(V, \omega)$. There exists one subadjoint variety corresponding to each complex simple Lie algebra,  as  listed in Table 1 of
\cite{Bu}. For example, the twisted cubic  curve in $\BP^3$ is the subadjoint variety corresponding to the simple Lie algebra of type $G_2$.
As subadjoint varieties are projectively normal, their affine cones  become
normal Legendrian subvarieties. }
\end{example}

\begin{example}\label{e.Bu} {\rm Landsberg-Manivel \cite{LM} and Buczynski (Chapters G, H, I of \cite{Bu2})  have discovered many examples of nonsingular Legendrian subvarieties in
$\BP V$, different from the subadjoint varieties of Example \ref{e.sub}. They are not projectively normal, so the Legendrian singularities of their affine cones are not normal.}
\end{example}

For later use, we recall the following proposition, which is just a reformulation of Corollary 5.5 and Lemma 5.6 of \cite{Bu}.

\begin{proposition}\label{p.QL}
Let $(V, \omega)$ be a symplectic vector space of dimension $2m$.
Fix a symplectic coordinate system $(y_1, \ldots, y_{2m})$ on $V$ satisfying $$\omega(\frac{\partial}{\partial y_{m+i}}, \frac{\partial}{\partial y_j}) = \delta_{ij}, \ 1 \leq i, j \leq m,$$ and define a homomorphism
$$\omega^{\dagger}: H^0(\BP V, \sO(2)) \longrightarrow  H^0(\BP V, T_{\BP V})$$ by sending a homogeneous quadratic polynomial $q(y_1, \ldots, y_{2m})$
     to the vector field on $\BP V$ given by the linear vector field $$\omega^{\dagger}(q) := \sum_{k=1}^m ( \frac{\partial q}{\partial y_{m+k}}
\frac{\partial}{\partial y_k} - \frac{\partial q}{\partial y_k} \frac{\partial}{\partial y_{m+k}}).$$ Let $\BP \sY$ be a Legendrian subvariety in $(\BP V, D^{\omega})$ as in Proposition \ref{p.projLegendre}.
Then the vector field $\omega^{\dagger}(q)$ is tangent to  $\BP \sY \subset \BP V$ if and only if $\sY$ is contained in the quadric hypersurface $q=0$.  \end{proposition}

\begin{proof}
The symplectic form $\omega$ gives an isomorphism $$\omega^{\sharp}: {\rm Sym}^2 V^* \cong \fsp(V) \subset \fsl(V)$$ such that for a symmetric bilinear form $Q \in {\rm Sym}^2 V^*$, the endomorphism $A:= \omega^{\sharp}(Q) \in \fsl(V)$ of $V$ sends $ v \in V$ to $A(v) \in V$ satisfying   $Q(v,u) = \omega(A(v), u)$ for all $u \in V$.  Corollary 5.5 and Lemma 5.6 of \cite{Bu} say that $\omega^{\sharp}$ identifies quadrics vanishing on $\sY$ with elements of $\fsp(V)$ which are tangent to $\sY$.  Our homomorphism $\omega^{\dagger}$ is, up to a scalar multiple,  just an expression of $\omega^{\sharp}$ in terms of linear vector fields on $V$, thus the proposition follows.
 \end{proof}

\section{Legendrian singularities with reduced tangent cones}\label{s.reduced}

In this section, we prove Theorem \ref{t.cone}.
The following local result in contact geometry is a key step of the proof.

\begin{theorem}\label{t.local}
In Notation \ref{n.standard}, let ${\bf m}_0$ be the maximal ideal of
the local ring $\sO_{\C^{2m+1}, 0}.$ Then the germ of a  hypersurface at $0$  defined by an equation   of the form $$x_{2m+1}  = h(x_1, \ldots, x_{2m}), \ h \in {\bf m}_0^3$$
is contactomorphic to the germ of the hyperplane $x_{2m+1} =0$ at $0$. \end{theorem}

To prove Theorem \ref{t.local}, we use the following two  classical results. The first one is Poincar\'e's result  on the normal forms of holomorphic vector fields (see Ch.4, Sec. 2.1 in \cite{AI}) and the second one is Arnold-Givental's relative Darboux theorem (see Ch. 4, Section 1.3, Theorem A in \cite{AG} or Theorem 1.1 in \cite{Zh}).

\begin{theorem}\label{t.Poincare}
Let $\vec{v}$ be a germ of holomorphic vector fields at the origin in $\C^d$ that vanishes at the origin.
Assume that the eigenvalues $(\lambda_1, \ldots,\lambda_d)$ of the linear part of $\vec{v}$ satisfy the non-resonant condition
that the convex hull of $\{ \lambda_1, \ldots, \lambda_d \} \subset \C$ does not contain $0$. Then $\vec{v}$ can be expressed as  a linear vector
field $\sum_{i=1}^d \lambda_i w_i \frac{\partial}{\partial w_i}$ under suitable holomorphic coordinates $(w_1, \ldots, w_d)$ in a
neighborhood of the origin in $\C^d$. \end{theorem}

\begin{theorem}\label{t.relDarboux}
In Notation \ref{n.standard}, let $X, H$ be two germs of complex submanifolds   at $0$ of the same dimension. Assume there exists a biholomorphic map $\varphi: X \to H$ such that
$$\varphi (0) = 0 \mbox{ and } \varphi^* \theta|_{H}  = g \cdot \theta|_{X}$$ for some nowhere-vanishing holomorphic function $g \in \sO_{X, 0}^*$ on $X$. Then the germ of $X$ at $0$ and the germ of $H$ at $0$ are contactomorphic. \end{theorem}

\begin{proof}[Proof of Theorem \ref{t.local}]
Put $f := h- x_{2m+1}$ and let $H$ be the germ at $0$ of the hypersurface defined by $f=0$.
The expression for $v^f$ in Theorem \ref{t.Kobayashi} gives \begin{eqnarray*} 2 v^f & =& \sum_{k=1}^m (\frac{\partial h}{\partial x_{m+k}} + x_{k}) \frac{\partial}{\partial x_k} + \sum_{k=1}^m (-\frac{\partial h}{\partial x_k} + x_{m+k}) \frac{\partial}{\partial x_{m+k}} \\ & & + \left( \sum_{k=1}^m (\frac{\partial h}{\partial x_k} x^k + \frac{\partial h}{\partial x_{m+k}} x_{m+k}) - 2f \right) \frac{\partial}{\partial x_{2m+1}}. \end{eqnarray*}
  Choose coordinates $z_i = x_i|_{H}, 1 \leq i \leq 2m,$ on $H$ such that
 \begin{equation}\label{e.z} \frac{\partial}{\partial z_i} =
  \frac{\partial}{\partial x_i} + \frac{\partial h}{\partial x_i} \frac{\partial}{\partial x_{2m+1}}, \ 1 \leq i \leq 2m. \end{equation}
Viewing $h|_H$ as a function $h(z_1, \ldots, z_{2m})$  and using (\ref{e.z}), the restriction of
$2 v^f$ to the hypersurface $H$ becomes 
\begin{equation}
\ \ 2 v^f|_H = \sum_{k=1}^m \left( (\frac{\partial h}{\partial z_{m+k}} + z_k) \frac{\partial }{\partial z_k} + (-\frac{\partial h}{\partial z_k} + z_{m+k}) \frac{\partial}{\partial z_{m+k}} \right). \end{equation}
   Since $h \in {\bf m}_0^3$,  the linear part of $2v^{f}|_H$ has eigenvalues
 $\lambda_1 = \cdots = \lambda_{2m} = 1.$ By Theorem \ref{t.Poincare}, we can find a new coordinate system $(w_1, \ldots, w_{2m})$ on a neighborhood of $0$ in
$H$ such that up to replacing $H$ by a smaller open subset, $$2v^{f}|_{H} = \sum_{k=1}^m ( w_k  \frac{\partial}{\partial w_k} + w_{m+k} \frac{\partial}{\partial w_{m+k}}).$$ In particular, the vector field $v^f$ does not vanish on $H \setminus 0$. By Theorem \ref{t.Kobayashi} (iii), this implies that the restriction $D|_{H \setminus 0}$ is  a vector subbundle of $T_{H\setminus 0}$ with rank $2m-1$.

Let $X$ be the germ at $0$ of the hyperplane $x_{2m+1} =0$ in $\C^{2m+1}$.
Consider the biholomorphic map $\varphi$ from  $X$  to $H$
defined by $\varphi^* w_i = x_i, 1 \leq i \leq 2m.$ Then $\varphi$ sends  the vector field $$2v^{-x_{2m+1}}|_X = \sum_{k=1}^m ( x_k  \frac{\partial}{\partial x_k} + x_{m+k} \frac{\partial}{\partial x_{m+k}})$$  to  the vector field $2v^f|_H$.  Let $\alpha: {\rm Bl}_0(X) \to X$ and
$\beta: {\rm Bl}_0(H) \to H$ be the blowups of the hypersurfaces at $0$. We have
submersions $$\alpha': {\rm Bl}_0(X) \to \BP D_0 \mbox{ and } \beta': {\rm Bl}_0(H) \to \BP D_0$$ whose fibers are the leaves of the radial vector fields
$v^{-x_{2m+1}}|_X$ and $v^{f}|_H$, respectively. By Theorem \ref{t.Kobayashi} (vi) and Lemma \ref{l.Cauchy}, the distribution $D|_{X \setminus 0}$ descends by $\alpha'$ to a contact structure $D^X$ on $\BP D_0$ and the distribution $D|_{H \setminus 0}$ descends by $\beta'$ to a contact structure $D^H$ on $\BP D_0$. Recall that any two contact structures on $\BP^{2m-1}$ are related by a projective linear transformation (Proposition 2.3 in \cite{LB2}). Thus by  a linear coordinate change of $w_1, \ldots, w_{2m}$, we may assume that $D^X = D^H$, i.e.,  the biholomorphism $\varphi_o := \varphi_{X\setminus 0}$ sends $D|_{X \setminus 0}$ to $ D|_{H \setminus 0}$. Then
$$\varphi_o^* (\theta|_{H \setminus 0} ) = g \cdot \theta|_{X \setminus 0}$$ for some nowhere-vanishing holomorphic function $g$ on $X \setminus 0$. By Hartogs extension, we can assume that $g$ is a nowhere-vanishing holomorphic function on $F$ such that $$\varphi^* (\theta|_{H} ) = g \cdot \theta|_{X}.$$ So the condition of Theorem \ref{t.relDarboux} is satisfied and
the germs of $X$ and $H$ at $0$ are contactomorphic.  \end{proof}

To prove Theorem \ref{t.cone}, we need the following two propositions.

\begin{proposition}\label{p.pTanCone}
Let $0 \in Z \subset (\C^{2m+1}, D=(\theta=0))$ be a Legendrian singularity whose projectivized tangent cone $\BP TC_{Z,0} \subset \BP T_{\C^{2m+1}, 0}$ is reduced.
Then $\BP TC_{Z,0}$ is contained in $\BP D_0 \subset \BP T_{\C^{2m+1}}.$  \end{proposition}

\begin{proof}
 Recall that the projectivized tangent cone  $\BP TC_{Z,0}$ is the exceptional divisor of the blowup
of $Z$ at $0$ (see e.g. \cite{Mu} Ch.3, Sec.3).  Thus we can realize each point of $\BP TC_{Z,0}$ as the limit of the
tangent lines to an arc $\{  x_t \in Z, t \in \Delta\}$, where $\Delta \subset \C$ is an open neighborhood of $0 \in \C$,  such that $x_0 = 0 \in Z$ and $x_t$ is a nonsingular point of $Z$ if $ t\in \Delta \setminus \{0\}$  (see Exercise 20.3 in \cite{Ha}). Since $Z$ is Legendrian, $T_{Z, x_t} \subset D_{x_t}$ for all $t \neq 0$. It follows that the limit of the tangent lines to the arc at $t=0$ is contained in $\BP D_0$.
Thus each (closed) point of $\BP TC_{Z,0}$ is contained in $\BP D_0$. As $\BP TC_{Z,0}$ is reduced, this implies the proposition. \end{proof}

\begin{proposition}\label{p.TanCone}
\extendspace{Let $0 \in Z \subset (\C^{2m+1}, D=(\theta=0))$ be a Legendrian singularity whose tangent cone} $TC_{Z,0} \subset T_{\C^{2m+1}, 0}$ is reduced. Then the tangent cone $TC_{Z,0}$ is contained in $D_0 \subset T_{\C^{2m+1}, 0}$ and is a Lagrangian cone with respect to the symplectic form $\omega= {\rm d} \theta |_{D_0}$. Consequently, the projectivized tangent cone $\BP TC_{Z,0}$ is a Legendrian subvariety of $\BP D_0$ with respect to $D^{\omega}$.  \end{proposition}

\begin{proof}
Let us regard all tangent spaces of $Z \subset \C^{2m+1}$ as affine subspaces in $\C^{2m+1}$.

Since $TC_{Z,0}$ is reduced, it is contained in $D_0$ from Proposition \ref{p.pTanCone}.
 To show that it is a Lagrangian cone, denote by $\sigma_x:= {\rm d} \theta|_{D_x}$  the symplectic form on $D_x \subset T_{\C^{2m+1},x}$ for each $x \in \C^{2m+1}$.
 To show that $TC_{Z,0}$ is a Lagrangian cone, it suffices by Proposition \ref{p.projLegendre} to show that $\sigma_0(u,v) =0$  for a general point $u \in TC_{Z,0}$ and any vector $$v \in T_{TC_{Z,0},u}
 \subset D_0, \ v \not\in \C u.$$
 Note that $\BP (\C u + \C v)$ is a tangent line to
  $\BP TC_{Z,0} \subset \BP D_0$ at the nonsingular point $[u] \in \BP TC_{Z,0}$.
  Let $\beta: {\rm Bl}_0(Z) \to Z$  be the blowup of $Z$ at $0$. Identify $\BP TC_{Z,0}$ with the exceptional divisor   of $\beta$ (see e.g. \cite{Mu} Ch.3, Sec.3). The assumption that
  $\BP TC_{Z,0}$ is reduced implies that ${\rm Bl}_0(Z)$ is  nonsingular at the  point $[u]
\in \BP TC_{Z,0}$.
Thus we can find an arc $\{ x_t \in Z,  t \in \Delta \}$ for a neighborhood $\Delta \subset \C$ of $0 \in \C$ such that \begin{itemize}
\item[(1)] $x_0 = 0 \in \C^{2m+1}$; \item[(2)] $x_t$ is a nonsingular point of  $Z$ if $ t \neq 0$;  and \item[(3)] the derivative
$\frac{\partial}{\partial t}|_{t=0} x_t \in \C u $. \end{itemize}
The vector $v$ gives a tangent vector $\vec{v}_0 \in T_{ \BP TC_{Z,0}, [u]} \subset T_{{\rm Bl}_0(Z), [u]}$.
Since both $\BP TC_{Z,0}$ and $ {\rm Bl}_0(Z)$ are nonsingular at $[u]$,
we can find a holomorphic family of tangent vectors $$\{\vec{v}_t \in T_{{\rm Bl}_0(Z), x_t}, t \in \Delta \setminus \{0\} \}$$ converging to $\vec{v}_0$. Let $$v_t = {\rm d} \beta(\vec{v}_t) \in T_{Z, x_t}, \ t\neq 0,$$ be the corresponding tangent vector to $Z$.  When $s \neq 0$, the plane $$\langle \frac{\partial}{\partial t}|_{t=s} x_t, v_s \rangle \ \subset \ D_{x_s}$$
is tangent to the smooth locus of $Z \subset \C^{2m+1}$. Thus $\sigma_{x_s}( \frac{\partial}{\partial t}|_{t=s} x_t, v_s) = 0$ for all $s \neq 0$ because $Z$ is Legendrian. Then by continuity, we obtain  $\sigma_0 (u, v) =0$.
\end{proof}

\begin{proof}[Proof of Theorem \ref{t.cone}]
It is immediate that the tangent cone of a Lagrangian cone $\sY \subset \C^{2m}$ at $0$ (in Definition \ref{d.Lagrange}) is isomorphic to itself. In particular, the tangent cone of a Lagrangian cone at $0$ is reduced. So one direction of Theorem \ref{t.cone} is trivial.

To prove the other direction, let us use the notation of Proposition \ref{p.TanCone}.
We are to show that the germ $0 \in Z \subset \C^{2m+1}$ is a Lagrangian cone singularity in the sense of Definition \ref{d.LagCone}, assuming that its tangent cone $TC_{Z,0}$ is reduced.

To start with,  we can assume that the tangent cone $ TC_{Z,0}$ spans $D_0$.
 For otherwise, the Zariski tangent space $T_{Z,0},$ which is the linear span of the tangent cone, does not contain $D_0$.  Then we can choose a nonsingular hypersurface $F$ containing $Z$
such that $T_{F,0} \neq D_0$ and apply Proposition \ref{p.product} to obtain a
a submersion $\psi: F \to M'$ to a manifold  of dimension $2m-1$ with a contact structure $D'$
    such that $\psi(Z)$ is a Legendrian subvariety of $(M', D')$ and $Z = \psi^{-1}(\psi(Z))$.
This implies that the tangent cone of $\psi(Z)$ at $\psi(0)$ is reduced. Thus by induction, we can assume that $\psi(0) \in \psi(Z)$ is  a Lagrangian cone singularity. Thus by Lemma \ref{l.conesing}, the Legendrian singularity $0 \in Z$ is a Lagrangian cone singularity. Thus from now, we assume that  $TC_{Z,0}$ spans $D_0$.

By Proposition \ref{p.TanCone}, the germ of $0 \in Z$ is contained in a germ of a nonsingular hypersurface $\Gamma$ with $T_{\Gamma,0} = D_0$.
We can view $\Gamma$ as the graph of
   an element of $\sO_{\C^{2m},0}$, where $\C^{2m} = (x_{2m+1}=0).$ In other words, in the notation of Theorem \ref{t.local},   there exists a homogeneous quadratic polynomial $q(x_1, \ldots, x_{2m})$ and a holomorphic function $h(x_1, \ldots, x_{2m}) \in {\bf m}_{0}^3$ such that
  $$x_{2m+1} = q(x_1, \ldots, x_{2m}) + h(x_1, \ldots, x_{2m})$$ is the defining equation of $\Gamma$. Set $$f := q(x_1, \ldots, x_{2m}) + h (x_1, \ldots, x_{2m}) - x_{2m+1}\  \in \ \sO_{\C^{2m+1}, 0}.$$ In terms of the coordinates $(z_1, \ldots, z_{2m})$ on $\Gamma$ defined by $z_i = x_{i}|_{\Gamma}, 1 \leq i \leq 2m,$
the same computation as in the proof of Theorem \ref{t.local} gives
\begin{eqnarray*}
\ \ 2 v^f|_{\Gamma} & =& \sum_{k=1}^m \left( (\frac{\partial h}{\partial z_{m+k}} + z_k) \frac{\partial }{\partial z_k} + (-\frac{\partial h}{\partial z_k} + z_{m+k}) \frac{\partial}{\partial z_{m+k}} \right)\\
& & + \sum_{k=1}^m \left( \frac{\partial q}{\partial z_{m+k}}  \frac{\partial }{\partial z_k} -\frac{\partial q}{\partial z_k} \frac{\partial}{\partial z_{m+k}} \right).    \end{eqnarray*}
Let ${\rm Bl}_0(\Gamma)$ be the blowup of $\Gamma$ at $0$.
Since $v^f|_{\Gamma}$ vanishes at $0$, it induces a vector field $\widetilde{v}$  on the blowup
${\rm Bl}_0(\Gamma)$. The first line in the expression of $2v^f|_{\Gamma}$ induces a vector field on ${\rm Bl}_0(\Gamma)$ that vanishes on the exceptional divisor. In fact, the vector field on the blowup induced by $$\sum_{k=1}^m \left( \frac{\partial h}{\partial z_{m+k}} \frac{\partial }{\partial z_k} - \frac{\partial h}{\partial z_k} \frac{\partial}{\partial z_{m+k}} \right)$$ vanishes on the exceptional divisor because
$h \in {\bf m}_0^3$, while the one induced by the radial vector field
$$\sum_{k=1}^m \left(  z_k \frac{\partial }{\partial z_k} + z_{m+k} \frac{\partial}{\partial z_{m+k}} \right)$$ obviously vanishes on the exceptional divisor.  The restriction of $\widetilde{v}$ to the exceptional divisor $\BP T_{\Gamma,0} \subset {\rm Bl}_0 (\Gamma)$ comes thus from the second line in the expression of $2 v^f|_{\Gamma}$, which is precisely
$\omega^{\dagger}(q)$ in the notation of Proposition \ref{p.QL}. Since $v^f$ is tangent to $Z$ by Theorem \ref{t.Kobayashi}, we know that $\widetilde{v}$ is tangent to ${\rm Bl}_0(Z) \subset {\rm Bl}_0(\Gamma)$ and also tangent to $\BP TC_{Z,0} \subset \BP T_{\Gamma,0} = \BP D_0$.
 Thus by Proposition \ref{p.QL} and Proposition \ref{p.TanCone},  the quadratic polynomial $q$ must vanish on the Legendrian subvariety $\BP TC_{Z,0}$ of $\BP D_0$.

 Let ${\bf m}_{\Gamma,0}$ be the maximal ideal of the local ring $\sO_{\Gamma,0}$ and let
 $J \subset \sO_{\Gamma,0}$ be the ideal of $Z$ inside $\Gamma$. Recall that the projective tangent cone $\BP TC_{Z,0} \subset \BP T_{\Gamma,0}$ is defined by the homogeneous ideal $J^* \subset {\rm Sym}^{\bullet} T_{\Gamma,0}^*$ generated by lowest order terms of elements of $J$.
 Since the projective scheme
$\BP TC_{Z,0} \subset \BP D_0$ is reduced, the homogeneous ideal $J^*$ defining the scheme $\BP TC_{Z,0}$ is a radical ideal, which implies that $q \in J^*.$
Since we have assumed that $ T_{Z,0}$ spans $D_0=  T_{\Gamma,0},$  the homogeneous ideal $J^*$ contains no elements of degree $1$.
Thus $q$ is an element of minimal degree in $J^*$. Consequently, it is the leading term of some element of $J$, i.e., there exists some $$c(z_1, \ldots, z_{2m}) \in {\bf m}^3_{\Gamma,0}$$ such that
$$q(z_1, \ldots, z_{2m}) - c(z_1, \ldots, z_{2m}) \ \in \ J.$$ Thus the hypersurface in $\C^{2m+1}$ defined by
$$x_{2m+1} = q+ h - (q - c) = c (x_1, \ldots, x_{2m}) + h(x_1, \ldots, x_{2m})
\in {\bf m}_0^3$$ contains $Z$. By Theorem \ref{t.local}, this hypersurface is contactomorphic to the hyperplane $x_{2m+1} =0$. Thus the germ at $0$ of $Z$ is contactomorphic to that
of a Lagrangian cone by Lemma \ref{l.Legendre}.
 \end{proof}

\section{Deformation-rigidity of normal Legendrian singularities}\label{s.normal}

In this section, we prove Theorem \ref{t.normal}.
Throughout, we fix a linear coordinate $t$ on $\C$ and denote a point of $\C$ simply by its coordinate $t \in \C$. Sometimes, we use another linear coordinate $\tau$ on $\C$ to distinguish it from $t$.
We begin with recalling a few standard facts on time-dependent vector fields. As some algebraic geometers may not be familiar with them, we will provide most of the proofs.

\begin{definition}\label{d.Kodaira} {\rm
Let $\Delta \subset \C$ be a neighborhood of $0 \in \C$. Let $M$ be a complex manifold.
Denote by $$\pi^M: M \times \Delta \to M \mbox{ and } \pi^{\Delta}: M \times \Delta \to \Delta$$ the natural projections.
\begin{itemize} \item[(i)]
When $W$ is a complex manifold and $\Psi: W \times \Delta \to M \times \Delta$ is a holomorphic map,
define for each $t \in \Delta$ and $w \in W$,  \begin{eqnarray*}
\Psi_t(w) & :=  & \pi^M \circ \Psi(w, t) \ \in M \\
\dot{\Psi}_t(w) & := & {\rm d} \pi^M \circ {\rm d} \Psi (\frac{\partial}{\partial t}|_{(w,t)}) \ \in T_{M, \Psi_t(w)}.
\end{eqnarray*} They define a holomorphic map $\Psi_t: W \to M$ and a holomorphic section
$\dot{\Psi}_t \in H^0(W, \Psi_t^* T_M).$
\item[(ii)] Let $\vec{B}$ be a holomorphic vector field on $M \times \Delta$
such that ${\rm d} \pi^{\Delta} (\vec{B}) =0$. For each $t \in \Delta$, denote by $\vec{B}_t \in H^0(M, T_M)$  the vector field
${\rm d} \pi^M \circ \vec{B}|_{M \times \{t\}}.$  \end{itemize}} \end{definition}

The following lemma on time-independent vector fields is a straightforward holomorphic translation of the standard result on differentiable manifolds (see e.g. Theorem 8.1 of \cite{St}).

\begin{lemma}\label{l.time-indep}
Let $X$ be a complex manifold and let $\vec{A} \in H^0(X, T_X)$ be a holomorphic vector field.
Then for each $x \in X$, there exist neighborhoods $O^x \subset X$ of $x$ and $\Delta^x \subset \C$ of $0$ with  a holomorphic map
$$\Phi^{\vec{A}}:  O^x \times  \Delta^x \to X \times \Delta^x$$ which is
biholomorphic over its image such that
$$\pi^{\Delta^x} \circ \Phi^{\vec{A}} = \pi^{\Delta^x}|_{O^x \times \Delta^x}, \
\Phi^{\vec{A}}|_{O^x \times \{0\}} = {\rm Id}_{O^x \times \{0\}} \mbox{ and }
{\rm d} \Phi^{\vec{A}} (\frac{\partial}{\partial \tau}) = \overrightarrow{A}
+\frac{\partial}{\partial \tau},$$
 where  $\overrightarrow{A}$ denotes the vector field on $X \times \Delta^x$ that is sent to $\vec{A}$ by the projection $\pi^{X}$ and  $\tau$ is the restriction of the coordinate $t$ to $\Delta^x \subset \C$. Moreover, the germ of $\Phi^{\vec{A}}$ at $(x,0)$ is uniquely determined by the above properties.
\end{lemma}

\begin{definition}\label{d.preserve} {\rm
Let $M$ be a complex manifold and let $\sD \subset T_M$ a subbundle.
We say that a vector field $\vec{A} \in H^0(M, T_M)$ {\em preserves} $\sD$ if
for each $x \in M$ and $\Phi^{\vec{A}}$ as in Lemma \ref{l.time-indep},
$${\rm d} \Phi_{\tau} (\sD) \subset \sD \mbox{ for each } \tau \in \Delta^x.$$
If $D$ is a contact structure on $M$, a vector field $\vec{A}$ on $M$ which preserves $D$ is
called an {\em infinitesimal contactomorphism}. The Lie algebra of all infinitesimal contactomorphisms of $(M,D)$ is denoted by $\cont(M,D)$.}
\end{definition}

We have the following  time-dependent contact version of Lemma \ref{l.time-indep}.

\begin{lemma}\label{l.time-dep}
Let $(M, D)$ be a contact manifold  and let $\pi^{\Delta}: M \times \Delta \to \Delta$ be the projection. Let
$\vec{B}$  be a holomorphic vector field on $M \times \Delta$ such that ${\rm d} \pi^{\Delta} (\vec{B}) =0$ and  $\vec{B}_t \in \cont(M,B) \subset H^0(M, T_M)$ in the terminology of Definition \ref{d.Kodaira} and Definition \ref{d.preserve}.
For each $z \in M$, there exist neighborhoods $U^z \subset M$ of $z$ and $\Delta^z \subset \Delta$ of $0$ with a holomorphic map  $$\Psi^{\vec{B}}: U^z \times \Delta^z  \to  M \times \Delta^z$$ which is biholomorphic over its image such that
\begin{itemize} \item[\rm (a)] $ \pi^{\Delta} \circ \Psi^{\vec{B}} = \pi^{\Delta}|_{U^z \times \Delta^z} $;
\item[\rm (b)] $\Psi^{\vec{B}}|_{U^z \times \{ 0\} } = {\rm Id}_{ U^z \times \{0\}}$;
 and
\item[\rm (c)] ${\rm d} \Psi^{\vec{B}} (\frac{\partial}{\partial t})  = \vec{B} + \frac{\partial}{\partial t}$ at points where both sides make sense; and
    \item[\rm (d)] the map $\Psi^{\vec{B}}_t: U^z \to M$ is contactomorphic over its image for each $t \in \Delta^z$.  \end{itemize}
    Moreover, the germ of $\Psi^{\vec{B}}$ at $(z, 0)$ is uniquely determined by the properties {\rm (a)}, {\rm (b)} and {\rm (c)}.
    \end{lemma}

\begin{proof} Equip $X :=  M \times \Delta$ (resp. $ X \times \Delta$) with the subbundle $\sD \subset T_X$ (resp. $\widetilde{\sD} \subset T_{X \times \Delta}$)  given
by $$\sD_{(z,t)} = ({\rm d} \pi^M|_{(z,t)})^{-1} (D_z) \ \left(\mbox{ resp. } \widetilde{\sD}_{(x, \tau)} = ({\rm d} \pi^X|_{(x, \tau)})^{-1} (\sD_x) \right),$$
 where $\pi^M: M \times \Delta \to M$ and $\pi^X: X\times \Delta \to X$ are the natural projections. Then the vector field $\vec{B}$ preserves $\sD$ and so do $\frac{\partial}{\partial t}$ and  $\vec{A}:= \vec{B} + \frac{\partial}{\partial t}$.   Applying Lemma \ref{l.time-indep} with $x= (z,0)$, we obtain two holomorphic maps $$\Phi^{\frac{\partial}{\partial t}},
\Phi^{\vec{A}}: \ O^x \times \Delta^x \to X \times \Delta^x $$ satisfying the properties in Lemma \ref{l.time-indep}, in particular,  \begin{equation}\label{e.push}
{\rm d} \Phi^{\frac{\partial}{\partial t}} ( \frac{\partial}{\partial \tau})  = \frac{\partial}{\partial t} + \frac{\partial}{\partial \tau} \  \mbox{ and } \
{\rm d} \Phi^{\vec{A}} (\frac{\partial}{\partial \tau}) = \vec{A} + \frac{\partial}{\partial \tau}. \end{equation}
Furthermore, \begin{equation}\label{e.D}
{\rm d} \Phi^{\frac{\partial}{\partial t}} ( \widetilde{\sD}) \subset \widetilde{\sD} \ \mbox{ and } \ {\rm d} \Phi^{\vec{A}} (\widetilde{\sD}) \subset \widetilde{\sD} \end{equation}
at points where they make sense.
 Let $$\gamma: M \times \Delta (= X ) \to M \times \Delta \times \Delta$$
be the diagonal embedding $\gamma (v, t) = (v, t, \tau=t)$. We have \begin{equation}\label{e.push2} {\rm d} \gamma (\frac{\partial}{\partial t})
= (\frac{\partial}{\partial t} + \frac{\partial}{\partial \tau}) |_{\gamma(X)}. \end{equation} and
\begin{equation}\label{e.D2} {\rm d} \gamma (\sD) \subset \widetilde{\sD}. \end{equation}
Set $$\Psi^{\vec{B}} := \pi^X \circ \Phi^{\vec{A}} \circ (\Phi^{\frac{\partial}{\partial t}})^{-1} \circ \gamma$$ which is defined on $U^z \times \Delta^z \subset M \times \Delta$
for a suitable choice of the neighborhoods $U^z$ and $\Delta^z$.  It is easy to see that the conditions (a) and (b) are satisfied. The condition (c) follows from (\ref{e.push}), (\ref{e.push2}) and
\begin{eqnarray*} {\rm d} \Psi^{\vec{B}} (\frac{\partial}{\partial t})
 &=& {\rm d} \pi^X \circ {\rm d} \Phi^{\vec{A}} \circ {\rm d} (\Phi^{\frac{\partial}{\partial t}})^{-1} (\frac{\partial}{\partial t} + \frac{\partial}{\partial \tau} ) \\
&=& {\rm d} \pi^X \circ {\rm d} \Phi^{\vec{A}} (\frac{\partial}{\partial \tau})  \\
&=& {\rm d} \pi^X (\overrightarrow{A} + \frac{\partial}{\partial \tau}) \\
&=& \vec{A} \ = \ \vec{B} + \frac{\partial}{\partial t}. \end{eqnarray*}
Finally, (d) is from ${\rm d} \Psi^{\vec{B}} ( \sD) \subset \sD$, which follows from (\ref{e.D}) and (\ref{e.D2}). The uniqueness of the germ of $\Psi^{\vec{B}}$ follows from the uniqueness theorem on solutions of ordinary differential equations. \end{proof}

The next lemma is a direct consequence of the uniqueness of the integral curve  of a vector field through a given point on a manifold, applied to the vector field $\vec{B} + \frac{\partial}{\partial t}.$

\begin{lemma}\label{l.map}
In Lemma \ref{l.time-dep}, assume that there exists a complex submanifold $W
 \subset M$  and a holomorphic map $F: W\times \Delta \to M \times \Delta,$ which is biholomorphic over its image, such that \begin{itemize} \item[\rm (i)] $\pi^{\Delta} \circ F = \pi^{\Delta}|_{F(W \times \Delta)} $;
 \item[\rm (ii)] $F|_{W \times \{0\} } = {\rm Id}_{W \times \{0\}}$; and \item[\rm (iii)]
${\rm d} F (\frac{\partial}{\partial t}) = \vec{B} + \frac{\partial}{\partial t}$ at points   where both sides make sense. \end{itemize}
For a point $z \in W \subset M,$ let $\Psi^{\vec{B}}: U^z \times \Delta^z \to M \times \Delta^z$ be as in Lemma \ref{l.time-dep}. Then, shrinking $U^z$ and $\Delta^z$ if necessary, we have
$$\Psi^{\vec{B}} ( (U^z \cap W)  \times \Delta^z) \subset F (W \times \Delta^z).$$
In other words, for any $y \in W$ in a neighborhood of $z$ and any $t \in \Delta$ in a neighborhood of $0$, we have
$\Psi^{\vec{B}}_t (y) = F_t(y)$. \end{lemma}

\begin{definition}\label{d.Kodaira2}
{\rm Let $M$ be a complex manifold and $\sS \subset M \times \Delta$ be a submanifold such that the restriction $\pi^{\Delta}|_{\sS}$ is submersive with connected fibers. Then we can view $\sS$ as a family of submanifolds $\{ S_t \subset M, t \in \Delta\}$ parametrized by $\Delta$. For each $ t \in \Delta$,  there exists a section $\dot{S}_t \in H^0(S_t, N_{S_t})$  of the normal bundle $N_{S_t}$ of $S_t \subset M$, called the {\em infinitesimal deformation} of the family  of submanifolds at $t$, which can be described in the following way for $t$ close to $0$ (this is a reformulation of the standard definition, e.g.,  pp. 148-150 of \cite{Kd}).

For a point $z \in S_0$, we can pick  neighborhoods $W_z \subset S_0$ of $z$ and
$\Delta_z \subset \Delta$ of $0$  with a holomorphic map
$G: W_z \times \Delta_z \to \sS \cap (M \times \Delta_z),$ which is biholomorphic over its image, such that
 \begin{itemize}
 \item[(1)] $G_t(w) \in S_{t}$, i.e., $G(w, t) \in \sS \cap (M \times \{ t \} )$
 for any $w \in W_z$ and $t \in \Delta_z$; and
  \item[(2)] $G|_{W_z \times \{0\}} = {\rm Id}_{W_z \times \{0\}}$, i.e., $G_0 = {\rm Id}_{W_z}$.
 \end{itemize}
  Then     for any $w \in W_z$ and $t \in \Delta_z,$  the value $\dot{S}_{t}$  of the section $\dot{S}_{t} \in H^0(S_{t}, N_{S_{t}})$ at the point $y = G_t(w) \in S_{t}$ is given by
   $$ \dot{S}_{t} (y) \ = \ {\rm d} \pi^M \circ \left( {\rm d} G(\frac{\partial}{\partial t}|_{(w,t)})  - \frac{\partial}{\partial t}|_{(y, t)} \right) \ \mbox{ modulo } \ T_{S_{t}, y}.  $$ This is independent of the choice of $G$. } \end{definition}

\begin{lemma}\label{l.submanifold}
In Definition \ref{d.Kodaira2}, assume that there exists  a vector field
 $\vec{B}$ on $M \times \Delta$ such that ${\rm d} \pi^{\Delta} (\vec{B}) =0$ and
$$\dot{S}_t (y) = \vec{B}_t(y) \mbox{ modulo } T_{S_t,y} \mbox{ for each  } t \in \Delta \mbox{ and
} y \in S_t.$$ For $z \in S_0 \subset M$, let $\Psi^{\vec{B}}: U^z \times \Delta^z \to M \times \Delta^z$  be the map defined in Lemma \ref{l.time-dep}.
Then $\Psi^{\vec{B}}_t(y) \in S_t$ for any $t \in \Delta^z$ close to $0$ and any $y \in S_0$ close to $z$.
 \end{lemma}

\begin{proof}
For a fixed $z \in S_0$,
let $G: W_z \times \Delta_z \to \sS$ be as in  Definition \ref{d.Kodaira2}.
  By the assumption on $\dot{S}_t$, we can find a  holomorphic vector field $\vec{E}$ on $W_z \times \Delta_z$ satisfying ${\rm d} \pi^{\Delta} (\vec{E}) =0$ and  $$ {\rm d} G (\vec{E}) =   \vec{B} - \left( {\rm d} G (\frac{\partial}{\partial t}) - \frac{\partial}{\partial t}\right),$$ in other words,
 \begin{equation}\label{e.G} {\rm d} G (\vec{E} + \frac{\partial}{\partial t}) = \vec{B} + \frac{\partial}{\partial t}, \end{equation}
 on the image of $G$.

 Let us apply Lemma \ref{l.time-dep} to the vector field $\vec{E}$ on $W_z \times \Delta_z$.
  We have neighborhoods $W \subset W_z$ of $z \in W_z$ and $\Delta^z \subset \Delta_z$ of $0 \in \Delta_z$ with a holomorphic map $\Psi^{\vec{E}}: W \times \Delta^z  \to W_z \times \Delta^z$ such that \begin{equation}\label{e.P} {\rm d} \Psi^{\vec{E}} (\frac{\partial}{\partial t})
= \vec{E} + \frac{\partial}{\partial t}. \end{equation}
Set $F:= G \circ \Psi^{\vec{E}} : W \times \Delta^z \to M \times \Delta^z$.
We claim that $F$ satisfies the conditions of Lemma \ref{l.map}.
The conditions (i) and (ii) of Lemma \ref{l.map} are immediate  from the properties of (1) and (2) of $G$ in Definition \ref{d.Kodaira2}, while
   (iii)  follows from
\begin{eqnarray*} {\rm d} F (\frac{\partial}{\partial t}) &=& {\rm d} G \circ {\rm d} \Psi^{\vec{E}} (\frac{\partial}{\partial t}) \\ &=& {\rm d} G ( \vec{E} + \frac{\partial}{\partial t}) \ \mbox{ by (\ref{e.P}) }\\
 &=& \vec{B} + \frac{\partial}{\partial t} \ \mbox{ by (\ref{e.G}). } \end{eqnarray*}
 By the claim, we can apply  Lemma \ref{l.map}
 to see $\Psi^{\vec{B}}_t(y) = F_t(y) \in S_t$ for any $t \in \Delta^z$ close to $0$ and any
 $y \in S_0$ close to $z \in S_0$.
    \end{proof}

To prove Theorem \ref{t.normal}, we recall some basic facts from contact geometry.
The following lemma is from Example 1.2.C in Chapter 4 of \cite{AG}.

\begin{lemma}\label{l.jet}
Let $S$ be a complex manifold and let $\sL$ be a line bundle on $S$. Then the underlying complex manifold of the 1-jet bundle $J^1_S \sL$ satisfying the exact sequence
$$ 0 \to T^*_S \otimes \sL \to J^1_S \sL \stackrel{j}{\to} \sL \to 0$$
has a natural contact structure. Moreover, a section $\Sigma \subset J^1_S \sL$ is a Legendrian submanifold if and only if it is the 1-jet
of the section $j(\Sigma) \in H^0(S, \sL)$. \end{lemma}

The next lemma is classical. See Theorem 7.1 in \cite{Kb} for (i) and  Lemma 7.1 in \cite{LB} for (ii).

  \begin{lemma}\label{l.contact1}
   For a contact manifold $(M, D)$, let $L = T_M/D$ be the quotient line bundle and $\vartheta: T_M \to L$ be the quotient homomorphism.

  \begin{itemize}
  \item[\rm (i)]  The homomorphism $$\cont(M,D) \subset H^0(M, T_M) \to H^0(M, L)$$ induced by  $\vartheta: T_M \to L$ gives an isomorphism of vector spaces $ \eta: \cont(M,D) \cong H^0(M,L).$
      \item[\rm (ii)] For a Legendrian submanifold  $S \subset M$, let  $\sL:= L|_{S}$ and let
       $j: J^1_S \sL \to \sL$ be as in Lemma \ref{l.jet}. Then there is a natural isomorphism $\zeta: N_{S} \cong J^1_{S} \sL$ between the  normal bundle $N_{S}$ of $S \subset M$  and the 1-jet bundle $J^1_{S} \sL$  such that $$ j \circ \zeta = \vartheta|_S : N_S \to \sL.$$ In particular,  the underlying manifold of $N_{S}$ has a natural contact structure via Lemma \ref{l.jet}.\end{itemize} \end{lemma}

          \begin{proposition}\label{p.contact2} Let us use the terminology of Lemma \ref{l.contact1}.
            \begin{itemize} \item[\rm (1)] Let $\{ S_t \subset M, t \in \Delta\}$ be a family of Legendrian submanifolds of $M$ with $S_0 = S$, parametrized by  a neighborhood $\Delta \subset \C$ of $0$. Let $$\vartheta (\dot{S}_0) \  \in \ H^0(S, \sL)$$ be the image of  the infinitesimal deformation $\dot{S}_0  \in H^0(S, N_{S})$ under the map induced by $\vartheta$, and
     let
      $$\zeta(\dot{S}_0) \
      \in \  H^0(S, J^1_{S} \sL)$$ be the image of $\dot{S}_0$ under the isomorphism $\zeta: N_S \cong J^1_S \sL$ of Lemma \ref{l.contact1} {\rm (ii)}.  Then $\zeta(\dot{S}_0)$ is  the 1-jet of $\vartheta (\dot{S}_0)$.
      \item[\rm (2)] For an element $\vec{A} \in \cont(M,D)$, let $$\eta(\vec{A})|_S \  \in \ H^0(S, \sL)$$ be the restriction of its image under $\eta$ in Lemma \ref{l.contact1} {\rm (i)}, and let
$$\zeta(\vec{A}) \ \in \ H^0(S, J^1_S \sL)$$ be the image of the element  $$(\vec{A} \mbox{ modulo } T_S ) \ \in \ H^0(S, N_S) $$ under $\zeta$ in Lemma \ref{l.contact1} {\rm (ii)}. Then $\zeta(\vec{A})$ is  the 1-jet of $\eta (\vec{A})|_S.$
      \end{itemize} \end{proposition}

\begin{proof}
For (1), we compute the infinitesimal deformation  $\dot{S}_0$  in a neighborhood
of $u \in S$ in the following way.
  Choose a biholomorphic map between a neighborhood $U^M \subset M$ of $u$ in $M$ and a neighborhood $U^N$ of $u$ in the normal bundle $N_S$ such that $S \cap U^M$ corresponds to the zero section of $N_{S \cap U^M}$ in $U^N$. We can shrink $\Delta$ and $U^M$, if necessary, so that the family of submanifolds $\{S_t, t \in \Delta\}$ gives a family of sections $s_t$ of the normal bundle $N_{S \cap U^M}$. Then $$\frac{{\rm d} s_t}{{\rm d} t}|_{t=0}
\in H^0(S \cap U^M, N_{S\cap U^M})$$ gives the value of $\dot{S}_{0}$ at $u$.
This is because we can choose $G$ in Definition \ref{d.Kodaira2} to respect the bundle structure on $U^M$ induced by the bundle structure of $N_{S}$.

Now by the relative Darboux Theorem (Theorem \ref{t.relDarboux} or  Theorem A in Chapter 4, Section 1.3  of \cite{AG}),  we may choose the biholomorphic map $U^M \cong U^N$ to be contactomorphic
with respect to the contact structure $D$ on $U^M \subset M$ and the contact structure on $U^N$ defined by $\zeta$ in  (ii) of Lemma \ref{l.contact1} combined with Lemma \ref{l.jet}. Then the sections $\{s_t, t \in \Delta\}$ corresponds to Legendrian sections of $N_S \cong J^1_S \sL$. Thus they are the 1-jets of sections $$\vartheta(s_t) = j \circ \zeta (s_t) \ \mbox{ of } \sL|_{S \cap U^M},$$ by Lemma \ref{l.jet}. Consequently, the derivative $ \frac{{\rm d} s_t}{{\rm d} t}|_{t=0},$ which represents $\dot{S}_0$,  is  the 1-jet of the section $$\vartheta(\frac{\partial s_t}{\partial t}|_{t=0}) = j \circ \zeta (\frac{\partial s_t}{\partial t}|_{t=0}) \ \mbox{ of } \sL|_{S \cap U^M}.$$  This proves (1).

(2) follows from (1) by considering the deformation of $S$ induced by the local contactomorphisms $\Phi^{\vec{A}}$ generated by $\vec{A}$ in the sense of Definition \ref{d.preserve}.
\end{proof}

\begin{proof}[Proof of Theorem \ref{t.normal}]
As the problem is local, we  replace $M$ by a Stein neighborhood of  $ x_0 \in Z_0$ in $M$
and assume that $M$ is Stein. Set $L = T_M/D$.

Let $S_t \subset Z_t$ be the smooth locus of the Legendrian subvariety $Z_t$. For each $ t \in \Delta$, we have the infinitesimal deformation $\dot{S}_{t}  \in H^0(S_t, N_{S_t})$
and the corresponding element $$\vartheta(\dot{S}_{t}) \in H^0(S_{t}, L|_{S_{t}})$$ from Proposition \ref{p.contact2} (1). As $Z_{t}$ is a normal variety, we have the extension
$a_{t} \in H^0(Z_{t}, L|_{Z_{t}})$ for each $t \in \Delta$ such that
$a_{t}|_{S_{t}} = \vartheta(\dot{S}_{t})$. Since $H^1(M, L \otimes \sI) =0$ for any ideal sheaf $\sI$ by our assumption that $M$ is Stein, we have
$b_t \in H^0(M, L)$ for each $t \in \Delta$ such that $a_t = b_t|_{Z_t}$.

We have $\vec{B}_t \in \cont(M,D)$ satisfying $\eta(\vec{B}_t)= b_t$ from Lemma \ref{l.contact1} (i) for each $t \in \Delta$.
The family of vector fields $\{ \vec{B}_t, t \in \Delta \}$ defines a vector field $\vec{B}$ on $M \times \Delta$ satisfying
${\rm d} \pi^{\Delta} (\vec{B}) =0$.
 Applying Lemma \ref{l.time-dep} to the vector field $\vec{B}$ on $M \times \Delta$ at the point
$x_0 \in M$, we obtain  $\Psi^{\vec{B}} : U^{x_0} \times \Delta^{x_0} \to M \times \Delta^{x_0}$ for some neighborhoods $x_0 \in U^{x_0} \subset M$ and $0 \in \Delta^{x_0} \subset \Delta$ such that  $\Psi^{\vec{B}}_t : U^{x_0} \to M$ is a contactomorphism over its image for each $t \in \Delta^{x_0}$.

We claim that $$ \dot{S}_t (y) = \vec{B}_t (y) \mbox{ modulo } T_{S_t, y}$$ for each $t \in \Delta$ and $y \in S_t$. By the isomorphism $ \zeta$ in Lemma \ref{l.contact1} (ii), we may check that   $ \zeta(\vec{B}_t|_{S_t}) = \zeta(\dot{S}_t) $ using the notation of Proposition \ref{p.contact2} (1). But  $\zeta(\dot{S}_t) $ is the 1-jet of $\vartheta (\dot{S}_t)$ by  Proposition \ref{p.contact2} (1) and  $ \zeta(\vec{B}_t|_{S_t})$ is the 1-jet of $\eta (\vec{B}_t|_{S_t})$ by  Proposition \ref{p.contact2} (2). Since
$$\eta(\vec{B}_t|_{S_t}) = b_t|_{S_t} = a_{t}|_{S_{t}} = \vartheta(\dot{S}_{t}),$$
we obtain  $ \zeta(\vec{B}_t|_{S_t}) =  \vartheta (\dot{S}_t)$. This proves the claim.

By the claim, we can apply Lemma \ref{l.submanifold} to conclude $\Psi^{\vec{B}}_t$ sends an open subset in $S_0$ into an open subset in $S_t$ for any $t$ close to $0$.  Consequently, the contactomorphism $\Psi^{\vec{B}}_t$ sends the germ of $Z_0$ at $x_0$ into  $Z_t$ for $t$ close to $0$. Setting $x_t = \Psi^{\vec{B}}_t(x_0)$, we obtain
    Theorem \ref{t.normal}. \end{proof}

\bibliographymark{References}

\providecommand{\bysame}{\leavevmode\hbox to3em{\hrulefill}\thinspace}
\providecommand{\arXiv}[2][]{\href{https://arxiv.org/abs/#2}{arXiv:#1#2}}
\providecommand{\MR}{\relax\ifhmode\unskip\space\fi MR }
\providecommand{\MRhref}[2]{%
  \href{http://www.ams.org/mathscinet-getitem?mr=#1}{#2}
}
\providecommand{\href}[2]{#2}

\end{document}